\documentclass[[amscd,amssymb,verbatim,12pt]{amsart}
\usepackage{amssymb, latexsym, amsmath, amscd, array, graphicx}
\usepackage{hyperref}
\usepackage[all]{xy}

\swapnumbers \numberwithin{equation}{section}
\usepackage{tikz-cd}

\theoremstyle{plain}

\newtheorem{thm}{Theorem}[section]
\newtheorem{theorem}[thm]{Theorem}

\newtheorem{lemma}[thm]{Lemma}

\newtheorem{prop}[thm]{Proposition}
\newtheorem{cor}[thm]{Corollary}

\theoremstyle{definition}
\newtheorem{defn}[thm]{Definition}

\newtheorem{rem}[thm]{Remark}

\newtheorem{ex}[thm]{Example}

\newtheorem{question}[thm]{Question}

\DeclareMathOperator{\cat}{{\mbox{\rm cat$_{\rm LS}$}}}

\DeclareMathOperator{\cd}{{\rm cd}}

\DeclareMathOperator{\TC}{{\rm TC}}
 
 \DeclareMathOperator{\idTC}{{\rm idTC}}
 \DeclareMathOperator{\vdTC}{{\rm vdTC}}
\DeclareMathOperator{\supp}{{\rm supp}}
\DeclareMathOperator{\dTC}{{\rm dTC}}

\def\cat{\protect\operatorname{cat}}
\def\acat{\protect\operatorname{acat}}

\def\C{{\mathbb C}}
\def\Z{{\mathbb Z}}
\def\Q{{\mathbb Q}}
\def\R{{\mathbb R}}

\def\1{\hbox{\rm\rlap {1}\hskip.03in{\rom I}}}
\def\Bbbone{{\rm1\mathchoice{\kern-0.25em}{\kern-0.25em}
{\kern-0.2em}{\kern-0.2em}I}}

\long\def\forget#1\forgotten{} %

\newcommand\ver[1]{\marginpar{\tiny Changed in Ver \VER}}

\date{\today}

\begin{document}

\title[Distributional TC]{Distributional Topological Complexity and LS-category}

\author[A.~Dranishnikov]{Alexander  Dranishnikov$^{1}$} 

\author[E.~Jauhari]{Ekansh Jauhari}

\thanks{$^{1}$Supported by Simons Foundation}

\address{Alexander N. Dranishnikov, Department of Mathematics, University
of Florida, 358 Little Hall, Gainesville, FL 32611-8105, USA.}
\email{dranish@ufl.edu}

\address{Ekansh Jauhari, Department of Mathematics, University
of Florida, 358 Little Hall, Gainesville, FL 32611-8105, USA.}
\email{ekanshjauhari@ufl.edu}


\begin{abstract}
We define a new version of Topological Complexity (TC) of a space, denoted as $\dTC$,
which, we think, fits better for motion planning for some autonomous systems. Like Topological complexity, dTC is also a homotopy invariant. Also, $\dTC$ has a corresponding analog, denoted as $d\cat$, to the Lusternik-Schnirelmann category (cat). In this paper, we do computations and estimates 
for both $\dTC(X)$ and $d\cat(X)$ for some spaces $X$ as well as a comparison with $\TC(X)$ and $\cat(X)$.

\end{abstract}



\maketitle

\section{Introduction}
Suppose that we have an autonomous mechanical system with the configuration space $X$. The system could be very complicated like Terminator 1 from the famous movie. Now we want to build an algorithm to get from a position $x\in X$ to a position $y\in X$ for each $x$ and $y$. Michael Farber noticed~\cite{Far1},\cite{Far2} that such an algorithm cannot be continuous as a function of two variables $(x,y)\in X\times X$ for most 
of the spaces $X$ and he introduced the topological complexity of $X$ as the minimal number of patches on $X\times X$ on which a continuous algorithm is possible.

In this paper, we propose a continuous algorithm on all $X\times X$ for more advanced autonomous systems, more like Terminator 2. The idea is that at position $x$, the system can break into finitely many pieces and each of the pieces travel to $y$ independently. Then at position $y$, all the pieces reassemble back into the system. This can be done continuously on $X\times X$.

Let us fix an upper bound $n$ on the number of pieces.
We assume that our system has mass one and for traveling from $x$ to $y$, it breaks into at most $n$ pieces such that the sum of the weights of the pieces is one.
We define {\em an $n$-distributed path} from $x$ to $y$ to be an unordered collection of $n$ paths in $X$ from $x$ to $y$ with 
non-negative weights such that the sum of the weights is one. We allow weights to be zero.
Thus, for each $x,y\in X$, our algorithm should give an $n$-distributed path from $x$ to $y$ which depends continuously on $(x,y)\in X\times X$. 

We are looking for a continuous algorithm in the following sense.
Let $\mathcal B(Z)$ denote the set of probability measures on a topological space $Z$ and $$\mathcal B_n(Z)=\{\mu\in\mathcal B(Z)\mid |\supp(\mu)| \le n\}$$ be the subspace of measures
supported by at most $n$ points. So, each such measure can be written as a formal sum  $$\mu=\sum_{z\in F\subset Z,\ |F|\le n}\lambda_zz,$$
where $\lambda_z\ge 0$ and $\sum\lambda_z=1$.
By definition, the support of $\mu$ is the following subset of $Z$: $\supp\mu=\{z\in Z\mid \lambda_z>0\}.$ 

In the case when $Z$ is a metric space, we consider the Lévy-Prokhorov metric on $\mathcal B_n(Z)$~\cite{P}. 
We identify the space $Z$ with a subspace of $\mathcal B_n(Z)$ by means of Dirac measures $\delta_z$ for every $z \in Z$.

Let $P(X)=\{f \hspace{1mm} | \hspace{1mm} f:[0,1]\to X\}$ be the space of all paths in $X$ and let $P(x,y)\subset P(X)$ be the subspace of all paths from $x$ to $y$.
Each element $\mu\in\mathcal B_n(P(x,y))$ is a  formal sum of paths from $x$ to $y$ with weights, so it is an  {\em $n$-distributed path} from $x$ to $y$.
Our desired continuous algorithm on $X$ is a continuous map $$s:X\times X\to \mathcal B_n(P(X))$$ such that
$s(x,y)\in \mathcal B_n(P(x,y))$ for all $x$ and $y$.

Given a space $X$, what is the smallest $n$ when the above continuous algorithm does exist? This question brings to life a new numerical invariant
which we call the {\em distributional topological complexity} of $X$.

The paper is organized as follows. In Section \ref{Classical numerical invariants}, we recall the Ganea-Schwarz approach to the classical numerical invariants: TC and cat. In Section \ref{New Numerical Invariants}, we define our new invariants: dTC and $d\cat$, and prove various properties for them, including their homotopy invariance and their relationships with each other and the classical invariants. In Section \ref{Lower Bounds}, we obtain sharp cohomological lower bounds for $d\cat$ and $\dTC$ using the cohomology of the symmetric products of $X$ and $X\times X$, respectively. Section \ref{Characterization} is devoted to the characterization of $d\cat$ and dTC to get parallel with Section \ref{Classical numerical invariants}. In Section \ref{Some Computations}, we compute dTC and $d\cat$ for various nice closed manifolds and obtain the formula for the dTC of the wedge-sum of aspherical complexes. We end this paper by proposing another numerical homotopy invariant in Section \ref{Epilogue}, which we call idTC.

\section{Classical Numerical Invariants}\label{Classical numerical invariants}
All the topological spaces considered here are path-connected metric ANR spaces. We recall classical definitions of the Lusternik-Schnirelmann category, $\cat (X)$, ~\cite{LuS},\cite{J},\cite{CLOT} and the topological complexity, $\TC(X)$, of a space $X$~\cite{Far1}, \cite{Far2}.

The {\emph{Lusternik-Schnirelmann category} (LS-category), $\cat (X)$, of  $X$ is the least number $n$ such that there is a covering  
$\{U_i\}$ of $X$ by $n+1$ open sets each of which is contractible in $X$.

The \emph{topological complexity}, $\TC(X)$, of  $X$ is the least number $n$  such that there is a covering $\{U_i\}$ of $X\times X$ by $n+1$ open sets over each of which there 
is a motion planning algorithm. We recall that a \emph{motion planning algorithm} over an open subset $U\subset X\times X$ is a continuous map $U\to P(X)$ that takes a pair $(x,y)$ to a path $s$ in $X$ with endpoints $s(0)=x$ and $s(1)=y$.

\subsection{Ganea-Schwarz's approach to LS-category.} Given a topological space $X$ with a fixed base point $x_{0} \in X$, let us define the space $$P_{0}(X) = \left\{\phi: [0,1] \to X \hspace{1mm}| \hspace{1mm} \phi(1) = x_{0}\right\}$$ with subspace topology inherited from the compact-open topology of $P(X)$. Let $p: P_{0}(X) \to X$ be the evaluation fibration $p: \phi \mapsto \phi(0)$. Then we define the $n^{th}$ Ganea space, denoted $G_{n}(X)$, to be the fiberwise join of  $(n+1)$-copies of $P_{0}(X)$ along $p$, i.e.,
$$
G_{n}(X) = \left\{\sum_{i=1}^{n+1} \lambda_{i} \phi_{i} \hspace{1mm} \middle| \hspace{1mm} \phi_{i} \in P_{0}(X), \sum_{i=1}^{n+1}\lambda_{i} = 1, \lambda_{i} \geq 0, \phi_{i}(0) = \phi_{j}(0)\right\},
$$
where each element is a formal ordered linear combination of based paths such that all terms where $\lambda_{i} = 0$ are dropped. Further, we define the $n^{th}$ Ganea fibration, $p_{n}^{X} : G_{n}(X) \to X$, as the fiberwise join of $(n+1)$-copies of the fibration $p$, i.e., $p_{n}^{X} (\sum_{i=1}^{n+1} \lambda_{i} \phi_{i}) = \phi_{i}(0)$, for any $i$ such that $\lambda_{i} > 0$. Then the following theorem gives the Ganea-Schwarz characterization of the LS-category~\cite{Sch}, \cite{CLOT}.

\begin{theorem}\label{original}
    For any $X$, $\textup{\text{cat}}(X) \leq n$ if and only if the fibration $p_{n}^{X}: G_{n}(X) \to X$ admits a section.
\end{theorem}
\subsection{Ganea-Schwarz's approach to TC}
Given a topological space $X$, let $\overline{p}: P(X) \to X\times X$ be the end-points fibration $\overline{p}: \phi \mapsto (\phi(0),\phi(1))$. Then we define the $n^{th}$ Schwarz-Ganea space, denoted $\Delta_{n}(X)$, to be the fiberwise join of  $(n+1)$-copies of $P(X)$ along $\overline{p}$, i.e.,
$$
\Delta_{n}(X) = \left\{\sum_{i=1}^{n+1} \lambda_{i} \phi_{i} \hspace{1mm} \middle| \hspace{1mm} \phi_{i} \in P(X), \sum_{i=1}^{n+1}\lambda_{i} = 1, \lambda_{i} \geq 0, \overline{p}\phi_{i}= \overline{p}\phi_{j}\right\},
$$
where each element is a formal ordered linear combination of  paths such that all terms where $\lambda_{i} = 0$ are dropped. We define the $n^{th}$ Schwarz-Ganea fibration, $\bar p_{n}^{X} : \Delta_{n}(X) \to X\times X$, as  $$\bar p_{n}^{X} \left(\sum_{i=1}^{n+1} \lambda_{i} \phi_{i}\right) = (\phi_{i}(0),\phi_{i}(1))$$ for any $i$ with $\lambda_{i} > 0$. Then the following theorem gives the Ganea-Schwarz characterization of the topological complexity \cite{Sch}.

\begin{theorem}\label{original2}
    For any $X$, $\TC(X) \leq n$ if and only if the fibration $\bar p_{n}^{X}: \Delta_{n}(X) \to X\times X$ admits a section.
\end{theorem}

\subsection{Lower bounds} The most useful lower bounds for $\cat(X)$ and $\TC(X)$ are given in terms of the cohomology of $X$ and $X\times X$, respectively. The cup-length of $X$ with coefficients in a ring $R$ is the maximal length $k$ of a non-zero product 
$$
\alpha_1\smile\alpha_2\smile\dots\smile\alpha_k\ne 0
$$
of cohomology classes of dimension $>0$. One can use the generalized cup product in the sense of ~\cite{Br}, still the cup-length will be a lower bound for $\cat(X)$~\cite{CLOT}.

In the case of $\TC$, the lower bound is given by the cup-length of zero-divisor cohomology classes of $H^*(X\times X;R)$~\cite{Far1}. We recall that a cohomology class in
$H^*(X\times X;R)$ is a zero-divisor if its restriction to the diagonal is zero.

\section{New Numerical Invariants}\label{New Numerical Invariants}
\subsection{Topology on $\mathcal B_n(Z)$}
Let $Z$ be a metric space. We denote by $C^\epsilon$ an open $\epsilon$-neighborhood of a subset $C\subset Z$. 
Let $\mu$ and $\nu$ be two measures in $\mathcal B_n(Z)$. Then by definition of the Lévy-Prokhorov metric~\cite{P}, 
$$\rho_Z(\mu,\nu)=\max\left\{\rho^\ell(\mu,\nu),\rho^r(\mu,\nu)\right\},$$ where
$$\rho^\ell(\mu,\nu)=\inf_C\{\epsilon>0\mid\mu(C)<\nu(C^\epsilon)+\epsilon\}$$ and
$$\rho^r(\mu,\nu)=\inf_C\{\epsilon>0\mid \nu(C)<\mu(C^\epsilon)+\epsilon\}.$$
Here, the infimum is taken over all measurable sets $C$. We note that only the sets $C\subset\supp\mu$ help to determine $\rho^\ell(\mu,\nu)$ and similarly, only the sets $C\subset\supp\nu$
help to determine $\rho^r(\mu,\nu)$.

When $Z$ is compact, the Lévy-Prokhorov metric defines the weak* topology. The same does the Wasserstein metric, but the Lévy-Prokhorov metric gives a better description of a basis at a point $\mu=\sum\lambda_zz$ of $\mathcal B_n(Z)$. Let $\delta=\min\{d(x,y)\mid x,y\in \supp\mu, \hspace{1mm} x\ne y\}$. Then for $\epsilon<\delta/2$, the open set $U(\mu,\epsilon)$, defined below as 
$$
U(\mu,\epsilon)=\left\{\sum\lambda'_xx \hspace{1mm} \middle| \hspace{1mm} \lambda_z<\sum_{x\in B(z,\epsilon)}\lambda'_x+\frac{\epsilon}{n}\right\}\subset B_\rho(\mu,\epsilon),
$$
lies in the open ball $B_\rho(\mu,\epsilon)$. The sets $U(\mu,\epsilon)$ form a basis of topology at $\mu$ in $\mathcal B_n(Z)$.

We consider the case when $Z=P(X)$ is the space of paths in a metric space $X$. We note that there is a continuous map $$\Phi:\mathcal B_n(P(X))\to P(\mathcal B_n(X))$$
defined as $\Phi(\mu)(t)=(\Phi(\sum\lambda_\phi\phi))(t)=\sum\lambda_\phi\phi(t)$.

\subsection{Distributional LS-category}
A $k$-distributed homotopy in $X$ between maps $f,g:Y\to X$ is a continuous map $H:Y\to \mathcal B_k(P(X))$ satisfying $H(y)\in\mathcal B_k(P(f(y),g(y)))$ for all $y \in Y$. A space $X$ is called {\em $k$-contractible} if there is a $k$-distributed homotopy between the identity map $1_X$ and a constant map $c:X\to x_0\in X$.
Clearly, a 1-distributed homotopy is just a homotopy and a 1-contractible space is just a contractible space.
\begin{defn}
{\em The distributional LS-category}, $d\cat(X)$, of a space $X$ is the minimal number $k$ such that $X$ is $(k+1)$-contractible.
\end{defn}

\begin{prop}\label{h.i.cat}
$d\cat$ is a homotopy invariant.
\end{prop}
\begin{proof}
We show that a homotopy domination $f:X\to Y$ implies the inequality $d\cat(X) \ge d\cat(Y)$. Let $g:Y\to X$ be the right inverse, i.e., $f g\sim 1_Y$.
Let $H:Y\times I\to Y$ be a homotopy between $1_Y$ and $fg$.
Let $d\cat(X) = k-1$ and $G:X\to\mathcal B_k(P_0(X))$ be a $k$-distributed contraction. Let us choose the base point $y_0=f(x_0)$. If $f_{*}: \mathcal{B}_{k}(P_{0}(X)) \to \mathcal{B}_{k}(P_{0}(Y))$ is induced by $f$ due to the functoriality of $\mathcal{B}_{k}$, then $$f_*Gg:Y\to\mathcal B_k(P_{0}(Y))$$ is a $k$-distributed homotopy between $fg$ and the constant map to $y_0$. We amend it by means of the homotopy $H$ as follows.
If $G(g(y))=\sum\lambda_\phi\phi$, we set $$F(y)=\sum\lambda_\phi (h_y\cdot f\phi)$$ where $h_y(t)=H(y,t)$ is the path from $y$ to $fg(y)$ and $h_y\cdot f\phi$ is the concatenation.
Then $$F:Y\to B_k(P_0(Y))$$ is a $k$-distributed contraction of $Y$ to $y_0$.
\end{proof}

\begin{prop}\label{cat}
$d\cat(X)\le\cat(X)$.
\end{prop}
\begin{proof}
Suppose that $\cat(X) = n$. Then by Theorem~\ref{original}, there is a section $s:X\to G_n(X)$ of the Ganea-Schwarz fibration $p_n$. Note that $s(x)=\sum_{i=1}^{n+1}\lambda_i\phi_i$
is the ordered sum where each $\phi_i$ is a path from $x$ to $x_0$. If we forget about the order, we obtain an $(n+1)$-distributed path from $x$ to $x_0$. Thus the section $s$ defines an $(n+1)$-distributed contraction of $X$ to $x_0$.
\end{proof}
We recall that for classical $\cat$ (as well as for $\TC$), there is the inequality
$$
\cat(X\times Y)\le\cat(X)+\cat(Y).$$
\begin{question}
Does the inequality $d\cat(X\times Y)\le d\cat(X)+d\cat(Y)$ hold true?
\end{question}
This looks like a difficult question. What is easy to prove is the following.
\begin{prop}
$d\cat(X\times Y)\le(d\cat(X)+1)(d\cat(Y)+1)-1$.
\end{prop}

The following proposition is straightforward:

\begin{prop}
$d\cat(X\vee Y)=\max\{d\cat(X), d\cat(Y)\}$.
\end{prop}

\begin{thm}\label{cover}
Let $p:X\to Y$ be a covering map and $X$ be path-connected. Then $d\cat(X)\le d\cat(Y)$.
\end{thm}
\begin{proof}
Let $y_0\in Y$ be a base point and let us fix some $x_{0} \in p^{-1}(y_{0})$ as the basepoint of $X$. For each $x\in p^{-1}(y_0)$, we fix a path $\gamma_x$ from $x$ to $x_0$.
Let $d\cat(Y) = k-1$ and $H:Y\to\mathcal B_k(Y)$ be a $k$-distributed contraction of $Y$ to $y_0$. We define a $k$-distributed contraction $G$ of $X$ to $x_0$ as follows. Given $x\in X$, let $y=p(x)$ and let
$H(y)=\sum_{\phi\in\supp H(y)}\lambda_\phi\phi$ where each $\phi$ is a path in $Y$ from $y$ to $y_0$. Let $\bar\phi$ denote the lift of $\phi$ beginning at $x$. Define $\hat\phi=\bar\phi\cdot\gamma_{\bar\phi(1)}$ to be the concatenation of $\bar\phi$ and $\gamma_{\bar\phi(1)}$. Note that $\hat\phi$ is a path in $X$ from $x$ to $x_0$. We define $G:X\to\mathcal B_k(P_{0}(X))$ by the formula
$$G(x)=\sum_{\phi\in\supp H(y)}\lambda_\phi\hat\phi.$$
\end{proof}

\subsection{Distributional Topological Complexity}
We recall that a {\em $k$-distributed navigation algorithm} on a space $X$ is a continuous map $$s:X\times X\to \mathcal B_k(P(X))$$ satisfying $s(x,y)\in \mathcal B_k(P(x,y))$ for all $(x,y) \in X \times X$.

\begin{defn} {\em The distributional topological complexity}, $\dTC(X)$, of a space $X$ is the minimal number $n$ such that $X$ admits an $(n+1)$-distributed navigation algorithm.
\end{defn}

\begin{prop}\label{hominvdTC}
$\dTC$ is a homotopy invariant.
\end{prop}
\begin{proof}
The proof is similar to that of Proposition~\ref{h.i.cat}. Again, we show that a homotopy domination $f:X\to Y$ implies the inequality $\dTC(X)\ge \dTC(Y)$. Let $g:Y\to X$ be the right inverse so that $f g\sim 1_Y$. Let $\dTC(X) = k-1$ and $H:Y\times I\to Y$ be a homotopy between $1_Y$ and $fg$.
Let $$s:X\times X\to\mathcal B_k(P(X))$$ be a $k$-distributed navigation algorithm. 
Then $$f_*s(g\times g):Y\times Y\to\mathcal B_k(P(Y))$$ takes any $(y,y') \in Y \times Y$ to a $k$-distributed path from $fg(y)$ to $fg(y')$. We amend it by means of the homotopy $H$ as follows.
If $s(g(y),g(y'))=\sum\lambda_\phi\phi$, we set $$\sigma(y,y')=\sum\lambda_\phi (h_y\cdot f\phi\cdot \bar h_{y'})$$ where $h_y(t)=H(y,t)$ is a path in $Y$ from $y$ to $fg(y)$ and $\bar h_y$
is its reverse.
Then $$\sigma:Y\times Y\to B_k(P(Y))$$ is a $k$-distributed navigation algorithm on $Y$.
\end{proof}

\begin{prop}\label{TC}
$\dTC(X)\le \TC(X)$.
\end{prop}
\begin{proof}
Suppose that $\TC(X) = n$. Then by Theorem~\ref{original2}, there is a section $\bar s:X\times X\to \Delta_n(X)$ of the Ganea-Schwarz fibration $\bar p_n$. Note that $\bar s(x,y)=\sum_{i=1}^{n+1}\lambda_i\phi_i$
is the ordered sum where each $\phi_i$ is a path from $x$ to $y$. If we forget about the order, we obtain an $(n+1)$-distributed path from $x$ to $y$. Thus, the section $\bar s$ defines an $(n+1)$-distributed navigation algorithm on $X$.
\end{proof}

The following two statements are easy to prove.

\begin{prop}\label{inequality}
$d\cat( X)\le \dTC(X)\le d\cat(X\times X)$.
\end{prop}

\begin{prop}\label{inequalityy}
$\max\{\dTC(X), \dTC(Y)\} \le \dTC(X \times Y)$.
\end{prop}

\begin{ex}\label{rpn}
$\dTC(\R P^n)=1$ for all $n$.
\end{ex}
\begin{proof}
We recall that points in $\R P^n$ are lines through the origin in $\R^{n+1}$. Two different lines $x$ and $y$ define a plane in $\R^{n+1}$. Let $\alpha$ and $\beta$ be the angles generated by our lines such that $\alpha+\beta=\pi$. There are two rotations $R_\alpha$ and $R_\beta$ of the line  $x$ to the line $y$ by the angle $\alpha$ and $\beta$, respectively. Each rotation is a path in $\R P^n$
from $x$ to $y$. We define a 2-distributed path from $x$ to $y$ by the formula $$s(x,y)=\frac{\beta}{\pi}R_\alpha+\frac{\alpha}{\pi}R_\beta.$$
We define $s(x,x)$ to be the constant path (the rotation with the angle zero). Clearly, the map $$s:\R P^n\times\R P^n\to\mathcal B_2(P(\R P^n))$$ is continuous.
Thus, $\dTC(\R P^n)\le 1$. Since $\R P^n$ is not contractible, $\dTC(\R P^n)=1$.
\end{proof}

\begin{rem}
    We note that in general, it is difficult to compute the exact values $\TC(\R P^{n})$~\cite{FTY}.
\end{rem}

We recall that a space $X$ with $\TC(X)=1$ for the classical topological complexity is homotopy equivalent to an odd-dimensional sphere~\cite{GLO}. However, this does not hold in the case of dTC: take $X = \mathbb{R}P^{n}$, $n > 1$.
\begin{thm}\label{topgroup}
For a topological group $X$, $\dTC(X)=d\cat(X)$.
\end{thm}
\begin{proof}
Let $d\cat(X) = k-1$ and $H:X\to\mathcal B_k(P_0(X))$ be a $k$-distributed contraction of the group $X$ to its unit $e\in X$.
We define a navigation algorithm
$$s:X\times X\to\mathcal B_k(P(X))$$ by the formula $s(x,y)=yH(y^{-1}x)$. Note that $s(x,y)$ is a $k$-distributed path  from $yy^{-1}x=x$ to $ye=y$. So, $\dTC(X) \le k-1 = d\cat(X)$.
\end{proof}
As in the case of classical $\cat$ and $\TC$~\cite{LS}, this result also holds for CW $H$-spaces.

\section{Lower Bounds}\label{Lower Bounds}
\subsection{Symmetric power}
The symmetric $m^{th}$ power of a space $X$ is defined as the orbit space $SP^m(X)=X^m/S_m$ of the action of the symmetric group $S_m$ on $X^m$ by permutations of coordinates.
Elements of $SP^m(X)$ are the orbits $[x_1,\dots, x_m]$ of $(x_1,\dots, x_m)\in X^m$. They can be considered as formal sums $\sum n_ix_i$, $x_i\in X$, $n_i\in\mathbb N$, $\sum n_i=m$, subject to the equivalence $kx+\ell x=(k+\ell)x$.
In this section, we regard $X$ as a subspace of $SP^{m}(X)$ under the diagonal embedding $\delta_{m}:X\to SP^m(X)$ which takes $x$ to the fixed point $[x,\dots,x]$ of the action of $S_m$ on $X^m$. Thus, $\delta_{m}(x)=mx$ by our convention. 

Let $x_0\in X$ be a base point. We denote by $\xi_m:X\to SP^m(X)$ the inclusion defined by the formula $\xi_m(x)=[x,x_0,\dots x_0]=x+(m-1)x_0$ and denote by $$\xi^m_{m+1}:SP^m(X)\to SP^{m+1}(X)$$ the inclusion defined as $\xi^m_{m+1}([x_1,\dots, x_m])=[x_1,\dots, x_m,x_0]$. The direct limit $$SP^\infty(X)=\lim_{m \to \infty}\{SP^m(X),\xi^m_{m+1}\}$$
is the free abelian topological monoid generated by $X$ with the addition defined as the addition of formal sums $\sum n_ix_i$ with $x_0$ being the zero element. Let $M(G,k)$ be a Moore space, i.e., a CW-complex having nontrivial reduced homology only in dimension $k$ and $H_k(M(G,k))=G$. The Dold-Thom theorem~\cite{DT} says that
the infinite symmetric power $SP^\infty$ takes Moore spaces to Eilenberg-MacLane spaces,
$SP^\infty(M(G,k))=K(G,k)$. The  Dold-Thom theorem can be extended to the following.
\begin{thm}\label{DT}
For any CW-complex $X$, there is a homomorphism of topological abelian monoids
$$
SP^\infty(X)\to\prod_k SP^\infty(M(H_k(X),k))$$ 
which is a homotopy equivalence.
\end{thm}

The following theorem is attributed to Steenrod (unpublished) and its proofs were given (independently) by Nakaoka~\cite{N} and Dold~\cite{D}.
\begin{thm}\label{Dold}
For a finite simplicial complex $X$ and any $m\in\mathbb N\cup{\infty}$, the inclusion homomorphism $(\xi_m)_*:H_*(X;R)\to H_*(SP^m(X);R)$ is a split monomorphism for any coefficient ring $R$.
\end{thm}
\begin{prop}\label{rational}
For a finite  simplicial complex $X$ and any $m\in\mathbb N$, the inclusion homomorphism $\delta_m^*:H^*(SP^m(X);\Q)\to H^*(X;\Q)$ is surjective.
\end{prop}
\begin{proof}
It suffices to show that $(\xi^m_{\infty})_*(\delta_m)_*:H_*(X;\Q)\to H_*(SP^\infty(X);\Q)$ is injective, where $\xi^m_\infty:SP^m(X)\to SP^\infty(X)$
is the natural inclusion defined by means of the base point $x_{0}$. We recall that $SP^\infty(X)$ is a free abelian monoid with $0=x_0$. Then $\xi^m_\infty\delta_m(x)=mx=m\xi_\infty(x)$ for all $x \in X$. Thus, $\xi^m_\infty\delta_m=m\xi_\infty$ and we need to show that
$$(m\xi_\infty)_*:H_*(X;\Q)\to H_*(SP^\infty(X);\Q)$$ is injective.
 Note that  $(m\xi_\infty)_*= m_*(\xi_\infty)_*$ where $$m:SP^\infty(X)\to SP^\infty(X)$$ is the multiplication by $m$ in $SP^\infty(X)$. 
In view of Theorem~\ref{Dold}, it suffices to show that $m_*$ is an isomorphism in rational homology. By virtue of  Theorem~\ref{DT}, there is a commutative diagram
$$
\begin{tikzcd}
    SP^\infty(X) \arrow{r}{\times m} \arrow[swap]{d} & SP^\infty(X) \arrow{d}
    \\
    \prod_k SP^\infty(M(H_k(X),k)) \arrow{r}{\times m} & \prod_k SP^\infty(M(H_k(X),k)) 
\end{tikzcd}
$$
where the vertical arrows are homotopy equivalences. In view of the Kunneth formula, it suffices to show that $$m:SP^\infty(M(H_k(X),k))\to SP^\infty(M(H_k(X),k))$$
induces an isomorphism of rational homology groups. Since each $H_k(X)$ is a finitely generated abelian group, the Kunneth formula reduces the problem to the case 
$$m:SP^\infty(M(C,k))\to SP^\infty(M(C,k))$$ where $C$ is a cyclic group. When $C$ is finite cyclic, then $m_*$ is an isomorphism of zero groups. 

We consider the case $m:SP^\infty(S^k)\to SP^\infty(S^k)$. Then the map $m$ is induced by a map $f:S^k\to S^k$ of degree $m$.
We recall that the cohomology ring $H^*(K(\Z,k);\Q)$ is generated by an element $a$ of dimension $k$. It is the polynomial algebra $\Q[a]$ when $k$ is even and the exterior algebra when $k$ is odd~\cite{FFG},~\cite{TD}. In both cases, $m:SP^\infty(S^k)\to SP^\infty(S^k)$ induces in rational cohomology a homomorphism which takes
$x$ to $mx$. Hence, $m$ induces an isomorphism $m^*$ in rational cohomology. Since rational homology is the dual of rational cohomology, $m_{*}$ is also an isomorphism.

We note that the reason why we switch from cohomology to homology and then back is that there is no Kunneth formula for cohomology for infinite CW complexes.
\end{proof}

\subsection{Lower bound for $d\cat$}

From now onwards, we will use the term \emph{inclusion} to refer to the diagonal embedding $\delta_{m}:X\to SP^m(X)$ that we defined above as $\delta_{m}(x) = [x,\ldots, x] = mx$.

\begin{lemma}\label{covercat}
    If $d\cat(X) < n$, then there exist sets $A_{1}, A_{2}, \ldots, A_{n}$ that cover $X$ such that for each $i$, the inclusion $A_{i} \to SP^{n!}(X)$ is null-homotopic.
\end{lemma}
\begin{proof}
Let $H:X\to \mathcal B_n(P_0(X))$ be an $n$-distributed homotopy of the identity $1_X$ to a constant map to $x_0$. We define $$A_i=\{x\in X\mid |\supp H(x)|=i\}.$$
For every $i$ and $x\in A_i$, we give the weight $\frac{n!}{i}$ to each $\phi\in\supp H(x)$ to obtain a path $f_{x}^{i}:I\to SP^{n!}(X)$ defined by the formula
$$
f_{x}^{i}(t)=\sum_{\phi \hspace{1mm} \in \hspace{1mm} \supp H(x)}\frac{n!}{i}\phi(t).
$$
Note that $f_x^{i}$ is a path in $SP^{n!}(X)$ from $n! \hspace{0.5mm} x$ to $n! \hspace{0.5mm} x_0$. Then the formula $H_{i}(x,t)=f_x^{i}(t)$ defines a homotopy $$H_{i}:A_i\times I\to SP^{n!}(X)$$
from the inclusion to the constant map.
\end{proof}

We refer to~\cite{Sr},\cite{GC} for the proof of the following
\begin{prop}\label{opencovercat}
Let $A\subset X$ be a subset of an ANR-space that is contractible in $X$. Then there is an open neighborhood $U$ of $A$ contractible in $X$.
\end{prop}
\begin{cor}\label{imppp}
If $d\cat(X) < n$, then there exist open sets $U_{1}, U_{2}, \ldots, U_{n}$ that cover $X$ such that for each $1 \leq i \leq n$, the set $U_i$ is contractible
in $SP^{n!}(X)$. 
\end{cor}

\begin{theorem}\label{boundcat}
 Suppose that $\alpha_{i}^{*} \in H^{k_{i}}\left(SP^{n!}(X);R\right)$, $1 \leq i \leq n$, for some ring $R$ and $k_{i} \geq 1$. Let  $\alpha_{i} \in H^{k_{i}}(X;R)$ be the image of $\alpha_{i}^{*}$ under the inclusion homomorphism such that  $\alpha_{1} \smile \alpha_{2} \smile \cdots \smile \alpha_{n} \neq 0$. Then $$d\cat(X) \geq n.$$
\end{theorem}
\begin{proof}
We proceed by contradiction. Assume that $d\cat(X) < n$. Then from Corollary~\ref{imppp}, we get open cover $\left\{U_{i}\right\}_{i=1}^{n}$ of $X$. We may assume that the sets $U_{i}$ are open in $SP^{n!}(X)$. Since $U_{i} \subset SP^{n!}(X)$ is 
contractible in $SP^{n!}(X)$, the inclusion homomorphism $$H^{k_i}(SP^{n!}(X);R)\to  H^{k_i}(U_i;R)$$ is trivial. The cohomology exact sequence of the pair $(SP^{n!}(X),U_i)$ implies that the homomorphism $$j_i^{*}:H^{k_i}(SP^{n!}(X),U_i;R)\to H^{k_i}(SP^{n!}(X);R)$$ is surjective. Let $j_i^*(\overline{\alpha_i^*})=\alpha_i^*$ and $\bar\alpha_i=\delta_{m}^{*}(\overline{\alpha^*_i})$.
Consider the following commutative diagram, where $j$ and $j'$ are sums of maps $j_{i}$ and $j_{i}'$, respectively, and $k = \sum_{i=1}^{n} k_{i}$.
\[ 
    \begin{tikzcd}[contains/.style = {draw=none,"\in" description,sloped}]
  H^{k}(X;R)  
& \arrow{l}{(j')^*} H^{k}\left(X,\bigcup_{i=1}^{n}(U_{i}\cap X);R\right)  
\\
H^{k}\left(SP^{n!}(X);R\right)  \arrow[swap]{u}{\delta^{*}}
& H^k\left(SP^{n!}(X),\bigcup_{i=1}^nU_{i};R\right) \arrow[swap]{u}{\delta^{*}} \arrow{l}{j^*} 
\\
\end{tikzcd}
\]
Since $\bar\alpha_{1}  \smile \bar\alpha_{2}  \smile \cdots \smile\bar \alpha_{n}\in H^k(X,X;R)=0$ and $$0\ne \alpha_{1} \smile \cdots \smile \alpha_{n}=(j')^*(\bar\alpha_{1}  \smile \cdots \smile\bar \alpha_{n}) =0,$$
we obtain a contradiction.
\end{proof}
\begin{cor}\label{rational cat}
The rational cup-length of  $X$ is a lower bound for $d\cat(X)$.
\end{cor}

\subsection{Lower bound for $\dTC$}
We recall that a deformation of a set $A\subset X$ to a subset $D\subset X$ is a homotopy $H:A\times I\to X$ such that $H|_{A\times\{0\}}:A\to X$ is the inclusion and $H(A\times\{1\})\subset D$.  
\begin{lemma}\label{coverdTC}
    If $\dTC(X) < n$, then there exist sets $A_{1}, A_{2}, \ldots, A_{n}$ that cover $X\times X$ such that for each $i$, the set $A_{i}$ is deformable in $SP^{n!}(X \times X)$
to the diagonal $\Delta X\subset X\times X\subset SP^{n!}(X\times X)$.
\end{lemma}
\begin{proof}
Let $s:X\times X\to \mathcal B_n(P(X))$ be an $n$-distributed navigation algorithm on $X$. As before, let us define $$A_i=\{(x,y)\in X \times X \mid |\supp s(x,y)|=i\}.$$ We define a homotopy $$H_{i}: A_{i} \times I\to SP^{n!}(X\times X)$$ by the formula
$$
H_{i}(x,y,t)=\sum_{\phi \hspace{1mm} \in \hspace{1mm} \supp s(x,y)}\frac{n!}{i}(\phi(t),y).
$$
Note that $H_{i}(x,y,0)=(x,y)$ and $H_{i}(x,y,1)=(y,y).$ Hence $H_{i}$ is a deformation of $A_{i}$ in $SP^{n!}(X\times X)$ to the diagonal.
\end{proof}

We note that
Proposition~\ref{opencovercat} can be extended by using the same technique to include deformations to any nice subspace, not necessarily a point.
\begin{prop}\label{opencoverTC}
Let $(X,D)$ be a CW complex pair and let $A\subset X$ be a subset that admits a deformation to $D$  in $X$. Then there is an open neighborhood $U$ of $A$ which admits a deformation to $D$ in $X$.
\end{prop}

\begin{cor}\label{openfordTC}
If $\dTC(X) < n$, then there exist open sets $U_{1}, U_{2}, \ldots, U_{n}$ that cover $X\times X$ such that for each $1 \leq i \leq n$, the set $U_i$ is deformable in $SP^{n!}(X\times X)$
to the diagonal $\Delta X\subset X\times X\subset SP^{n!}(X\times X)$.
\end{cor}
We recall that an element $\alpha\in H^*(X\times X;R)$ is called a zero divisor~\cite{Far1} if $\Delta^*(\alpha)=0$ where $\Delta :X\to X\times X$ is the inclusion of the diagonal.
\begin{theorem}\label{boundtc}
Suppose that $\beta_{i}^{*} \in H^{k_i}\left(SP^{n!}(X \times X);R\right)$, $1 \le i \le n$, for some ring $R$ and $k_{i} \ge 1$, are cohomology classes 
such that their images $\beta_i\in H^{k_i}(X\times X;R)$ under the inclusion homomorphism are zero-divisors and the cup product
 $\beta_{1} \smile \beta_{2} \smile \cdots \smile \beta_{n} \neq 0$. Then $$\dTC(X) \geq n.$$
\end{theorem}
\begin{proof}
Assume the contrary that $\dTC(X) < n$. By Corollary~\ref{openfordTC} there is an open cover $\{U_i\}_{i=1}^n$ of $X\times X$ such that each $U_i$ is deformable
 in $SP^{n!}(X\times X)$ to $\Delta X$. We may assume that $U_i$ are open in $SP^{n!}(X\times X)$. Since the inclusion $U_i\to SP^{n!}(X 
\times X)$ is homotopic to
the composition $$U_i \xrightarrow{\text{proj}_{2}} X \xrightarrow{\Delta} \Delta X\stackrel{\subset}\to X\times X\stackrel{\subset}\to SP^{n!}(X\times X)$$ and $\beta_i$ is a zero-divisor, it follows that $\beta_i^*$ goes to zero
under the inclusion homomorphism $$H^{k_i}(SP^{n!}(X\times X);R)\to H^{k_i}(U_i;R).$$ It follows from the exact sequence of pair $(SP^{n!}(X \times X), U_{i})$ that there exists $$\overline{\beta_i^*}\in H^{k_i}(SP^{n!}(X\times X),U_i;R)$$ which maps to $\beta_i^*$. For $k = \sum k_{i}$, we have the commutative diagram:
$$
\begin{tikzcd}
H^k(SP^{n!}(X\times X),\cup_iU_i;R) \arrow{r} \arrow[swap]{d} & H^k(SP^{n!}(X\times X);R) \arrow{d}
\\
H^k(X\times X,(\cup_i(U\cap (X\times X));R) \arrow{r} & H^k(X\times X;R)
\end{tikzcd}
$$
The cup product from the top-left corner $\overline{\beta_{1}^*} \smile \overline{\beta_{2}^*} \smile \cdots \smile \overline{\beta_{n}^*}$ goes to the non-zero cup product
 $\beta_{1} \smile \beta_{2} \smile \cdots \smile \beta_{n}$. On the other hand, it factors through the group $$ H^k(X\times X,(\cup_i(U\cap (X\times X));R) =H^k(X\times X,X\times X;R)=0.$$
Hence, we have arrived at a contradiction.
\end{proof}
\begin{cor}\label{rational dTC}
The rational zero-divisor cup-length of  $X\times X$ is a lower bound for $\dTC(X)$.
\end{cor}

\section{Ganea-Schwarz's Characterization of $d\cat$ and $\dTC$}\label{Characterization}

The space $\mathcal B_n(X)$ is defined in the Introduction.

For any map $p:E\to B$, we define the map $\mathcal B_n(p):E_n\to B$ as follows. Let $$E_n=\{\mu\in\mathcal B_n(E)\mid \supp\mu\subset p^{-1}(x), x\in B\}$$ 
denote the result of the fiberwise application of the functor $\mathcal B_n$ to $E$.
Thus, $$E_n=\cup_{x\in B} \hspace{1mm} \mathcal B_n(p^{-1}(x))\subset\mathcal B_n(E).$$  We define
$\mathcal B_n(p)(\mu)=x$ for $\mu\in \mathcal B_n(p^{-1}(x))$.

\begin{prop}
Let $p:E\to B$ be a  Hurewicz fibration. Then $\mathcal B_n(p):E_n\to B$ is a  Hurewicz fibration.
\end{prop}
\begin{proof}
Let $H:X\times I\to B$ be a homotopy with an initial lift $\mu:X\to E_n$. Let $$Y=\{(x,z)\in X\times E\mid z\in \supp(\mu(x))\}\subset X\times E$$ and let $\pi: Y\to X$ be the projection. Since $p$ is a fibration, the homotopy $$H\circ (\pi\times 1_I):Y\times I\to B$$ has a covering homotopy
$F:Y\times I\to E$ with the initial lift $\bar h:Y\to E$ defined as $\bar h(x,z)=z$. We define $\bar H:X\times I\to E_n$ as $$\bar H(x,t)=(F|_{Y\times t})_*(\mu(x))$$ where 
$(F|_{Y\times t})_*(\mu(x))$ is the push-forward map on measures in $\mathcal B_n(p^{-1}(x))$. This map is continuous as the restriction of the push-forward map $F_*$.
Note that $\bar H(x,0)=\mu(x)$.
\end{proof}
Let $F$ be the fiber of a fibration $p$, then $\mathcal B_n(F)$ is the fiber of the fibration $\mathcal B_n(p)$.

We note that if $F$ is $r$-connected for $r>0$, then  $\mathcal B_n(F)$ is $(2n+r-2)$-connected~\cite{KK}.
\begin{prop}\label{discrete}
If $F$ is discrete, then $\mathcal B_n(F)$ is $(n-2)$-connected.
\end{prop}
We refer to~\cite{KK} for the proof when $F$ is finite and note that the case of infinite $F$ can be derived from it.

\subsection{Characterization of $d\cat$ and $\dTC$ }

Recall that $P_0(X)$ is the space of paths in $X$ issued to $x_0\in X$ and $p_{0}:P_0(X)\to X$ is the evaluation fibration $p_{0}(f)=f(0)$.
\begin{prop}\label{chardcat}
$d\cat(X) \le n$ if and only if the fibration $$\mathcal B_{n+1}(p_{0}):P_0(X)_{n+1}\to X$$ admits a section.
\end{prop}
\begin{proof}
Let $s:X\to P_0(X)_{n+1}$ be a section. Then $s(x)$ is an $(n+1)$-distributed path from $x$ to $x_0$. Thus, $s:X\to\mathcal B_n(P_0(X))$ is an $(n+1)$-contraction of $X$ to $x_0$. In the other direction, an $(n+1)$-contraction $H:X\to\mathcal B_{n+1}(P_0(X))$ defines a section of $\mathcal B_{n+1}(p_{0})$.
\end{proof}

Recall that $P(X)$ is the space of all paths in $X$ and $p: P(X)\to X\times X$ is the end-points fibration $p(f)=(f(0),f(1))$.
\begin{prop}\label{chardTC}
$\dTC (X)\le n$ if and only if the fibration $$\mathcal B_{n+1}(p):P(X)_{n+1}\to X\times X$$ admits a section.
\end{prop}
\begin{proof}
Let $s:X\times X\to P(X)_{n+1}$ be a section. Then $s(x,y)$ is an $(n+1)$-distributed path from $x$ to $y$. Thus, $s:X \times X\to\mathcal B_{n+1}(P(X))$ is an $(n+1)$-distributed navigation algorithm. In the other direction, an $(n+1)$-navigation algorithm on $X$ defines a section of $\mathcal B_{n+1}(p)$.
\end{proof}

\section{Some Computations}\label{Some Computations}

\subsection{Some computations of $d\cat$}

\begin{prop}\label{orient}
For any orientable surface $\Sigma_{g}$ of genus $g>0$, $d\cat (\Sigma_{g})=2$.
\end{prop}
\begin{proof}
By Proposition~\ref{cat}, $d\cat (\Sigma_{g})\le\cat(\Sigma_{g})= 2$. Since the rational cup-length of $\Sigma_{g}$ equals 2 when $g > 0$, by Corollary~\ref{rational cat}, we have $2\le d\cat(\Sigma_{g})$.
\end{proof}

\begin{prop}\label{rpndcat}
    $d\cat(\R P^n)=1$ for all $n$.
\end{prop}
\begin{proof}
    This follows from Proposition~\ref{inequality} and Example~\ref{rpn}.
\end{proof}

We recall that any space $X$ with $\cat(X)=1$ for the classical LS-category is a co-H-space~\cite{J}. However, this does not hold in the case of $d\cat$: take $X = \mathbb{R}P^{n}$ for $n > 1$.

\begin{rem}
    We also note that dTC and $d\cat$ agree for $\R P^{2}$ which is neither homotopy equivalent to a $H$-space, nor to a (product of) co-$H$-space(s). An example of such a space for which the classical invariants TC and cat agree is not known.
\end{rem}

\begin{prop}\label{nonorient}
For a non-orientable surface $N_g$ of genus $g>1$, $d\cat(N_g)=2$.
\end{prop}
\begin{proof}
For each $N_g$ with $g>1$, there is an orientable cover $p:M_h\to N_g$ with $h>0$. Then by Proposition~\ref{orient}, Theorem~\ref{cover}, and Proposition~\ref{cat}, we obtain
$2=d\cat(M_h) \le d\cat(N_g) \le \cat(N_g) =2$.  
\end{proof}

\begin{rem}
The connected sum formula~\cite{DS1},\cite{DS2} $$\cat(M\# N)=\max\{\cat(M),\cat(N)\}$$ does not hold for $d\cat$ since by Proposition~\ref{nonorient},
$$2=d\cat(\R P^2\#\R P^2)\ne d\cat(\R P^2)=1.$$
\end{rem}

\begin{prop}
$d\cat(\C P^n)=n$.
\end{prop}
\begin{proof}
Since the rational cup-length of $\C P^n$ is $n$, by Corollary~\ref{rational cat} and Proposition~\ref{cat}, we obtain
$ n\le d\cat(\C P^n)\le\cat(\C P^n)=n$. 
\end{proof}

\begin{prop}\label{spheres}
For the product of spheres $$d\cat\left(\prod_{i=1}^m S^{n_i}\right)=m.$$
\end{prop}
\begin{proof}
Since the rational cup-length of $\prod_{i=1}^m S^{n_i}$ is $m$, in view of Corollary~\ref{rational cat} and Proposition~\ref{cat}, we obtain
$$
m\le d\cat\left(\prod_{i=1}^m S^{n_i}\right)\le  \cat\left(\prod_{i=1}^m S^{n_i}\right)=m.$$
\end{proof}

\subsection{Some computations of $\dTC$}

\begin{prop}\label{tcsph}
$\dTC(S^n)=1$ if $n$ is odd and $\dTC(S^n)=2$ if $n$ is even.
\end{prop}
\begin{proof}
The case of odd $n$ is a direct consequence of the inequality $\dTC(S^n)\le \TC(S^n)=1$ from Proposition~\ref{TC}.
For $n$ even, we apply Corollary~\ref{rational dTC} and the fact that the rational zero-divisor cup-length of $S^n\times S^n$ is $2$ \cite{Far1}.
Thus,
$$2\le\dTC(S^n)\le \TC(S^n)=2.$$
\end{proof}

\begin{prop}
$\dTC(\Sigma_g)=4$ if $g>1$ and $\dTC(\Sigma_1) = \dTC(\Sigma_0)=2$.
\end{prop}
\begin{proof}
Since the rational zero-divisor cup-length of $\Sigma_g\times \Sigma_g$ is $4$~\cite{Far1}, by Corollary~\ref{rational dTC} and Proposition~\ref{TC}, we obtain
$$ 4\le \dTC(\Sigma_g)\le\TC(\Sigma_g)=4\ \ \text{ when}\ \ g > 1.$$ 
For torus $\Sigma_{1}$, the result follows immediately from Theorem~\ref{topgroup} and Proposition~\ref{orient}. For $2$-sphere $\Sigma_0$, it follows from Proposition~\ref{tcsph}.
\end{proof}
\begin{question}
What is $\dTC(N_g)$ of non-orientable surfaces? In particular, what is $\dTC(K)$ of the Klein bottle?
\end{question}

\begin{prop}
$\dTC(\C P^n)=2n$.
\end{prop}
\begin{proof}
Since the rational zero-divisor cup-length of $\C P^n\times\C P^n$ is $2n$~\cite{Far2}, by Corollary~\ref{rational dTC} and Proposition~\ref{TC}, we obtain
$$2n\le \dTC(\C P^n)\le\TC(\C P^n)=2n.$$ 
\end{proof}

\begin{prop}
    If $X$ is a connected finite graph, then $\dTC(X) = 1$ if $\beta_{1}(X) = 1$ and $\dTC(X) = 2$ if $\beta_{1}(X) > 1$. 
\end{prop}

\begin{proof}
    If $\beta_{1}(X) = 1$, then $X\simeq S^{1}$, so $\dTC(X) = 1$ by Proposition \ref{tcsph}. If $\beta_{1}(X) > 1$, then the rational zero-divisor cup-length of $X \times X$ is 2 \cite{Far1}. So, by Corollary \ref{rational dTC}, we get $2 \leq \dTC(X) \leq \TC(X) = 2$. 
\end{proof}

We note that the lower bounds obtained in Theorems \ref{boundcat} and \ref{boundtc} are sharp. Also, the inequalities in Propositions \ref{cat}, \ref{TC}, \ref{inequality}, and \ref{inequalityy} can be strict.

\subsection{$\dTC$ of the wedge-sum}

\begin{thm}\label{wedge}
The equality
$$\dTC(X\vee Y)=\max\{\dTC (X),\dTC (Y),d\cat(X\times Y)\}$$
holds for aspherical complexes $X$ and $Y$
whenever $$\max\{\dim X,\dim Y\}<\max\{\dTC (X),\dTC (Y),d\cat(X\times Y)\}.$$
\end{thm}
\begin{proof}
 Let $v\in X\vee Y$ be the wedge point and let $r_X:X\vee Y\to X$ and $r_Y:X\vee Y\to Y$ be the retractions collapsing $Y$ to $v$ and $X$ to $v$, respectively. The inequality $$\dTC(X\vee Y)\ge\max\{\dTC (X),\dTC (Y)\}$$ holds in view of retractions $r_X$ and $r_Y$.
Let $\dTC(X\vee Y)=n$. We show that $d\cat(X\times Y)\le n$.
Let $$s:(X\vee Y)\times(X\vee Y)\to\mathcal B_{n+1}(P(X\vee Y))$$ be a $(n+1)$-distributed
navigation algorithm. We define an $(n+1)$-contraction $$H:X\times Y\to \mathcal B_{n+1}(P_0(X\times Y))$$ of $X\times Y$ to the point $(v,v)$ by the formula
$$
H(x,y)=\sum_{\phi\in\supp s(x,y)}\lambda_\phi(r_X\phi,r_Y\bar\phi).
$$
Here $\phi:I\to X\vee Y$ is a path from $x$ to $y$ and $\bar\phi$ is the reverse path from $y$ to $x$. Clearly, $r_X\phi:I\to X$ is a path from $x$ to $v$ and $r_Y\bar\phi:I\to Y$ is a path from $y$ to $v$. Hence, $(r_X\phi,r_Y\bar\phi):I\to X\times Y$ is a path from $(x,y)$ to $(v,v)$.

Next, we prove the inequality
$$\dTC(X\vee Y)\le\max\{\dTC (X),\dTC (Y),d\cat(X\times Y)\}.$$
Let $n=\max\{\dTC (X),\dTC (Y),d\cat(X\times Y)\}$. We construct a section $s$ of the fibration $\mathcal B_{n+1}(p)$ defined for the endpoints fibration $$p:P(X\vee Y)\to(X\vee Y)\times(X\vee Y)=(X\times X)\cup (Y\times Y)\cup (X\times Y)\cup (Y\times X).$$
Let $c_v:I\to X\vee Y$ denote the constant loop $c_v(t)=v$. We define $s(v,v)$ to be the Dirac measure $\delta_{c_v}$ supported by $c_v$. 
Note that the fibration $\mathcal B_{n+1}(p^X) :P(X)_{n+1}\to X \times X$ defined for the end-points fibration 
$$
p^{X}:P(X)\to X\times X
$$ 
is naturally embedded in the fibration $\mathcal B_{n+1}(p)$. Since $\dTC(X)\le n$, by Proposition~\ref{chardTC}, there is a section $$s_X:X\times X\to P(X)_{n+1}$$ of  $\mathcal B_{n+1}(p^X)$,
which is also a section of $\mathcal B_{n+1}(p)$. We may assume that $s_X(v,v)=s(v,v)$. Similarly, there is a section $$s_Y:Y\times Y\to P(X\vee Y)_{n+1}$$ of $\mathcal B_{n+1}(p)$
with $s_Y(v,v)=s(v,v)$. 

Let $(v,v)\in X\times Y$ be the base point. We define a map $$h: P_0(X\times Y)\to P(X\vee Y)$$ by the formula $h(\phi)=\phi_X\cdot\bar\phi_Y$ where $\phi=(\phi_X,\phi_Y):I\to X\times Y$.
Note that $h$ commutes with the projections 
$$p_0(\phi)=(\phi_X(0),\phi_Y(0))=(\phi_X\cdot\bar\phi_Y(0),\phi_X\cdot\bar\phi_Y(1))=p(\phi_X\cdot\bar\phi_Y)= p(h(\phi)).$$
Therefore, $h$ defines a fiberwise map $$h_{n+1}:P_0(X\times Y)_{n+1}\to P(X\vee Y)_{n+1}.$$
Since $d\cat (X\times Y)\le n$, by Proposition~\ref{chardcat}, there is a section $\sigma_{X\times Y}$ of $\mathcal B_{n+1}(p_0^{X\times Y})$.
Then $$s'_{X\times Y}=h_{n+1} \hspace{1mm} \sigma_{X\times Y} : X \times Y \to P(X \vee Y)_{n+1}$$ is a section of $\mathcal B_{n+1}(p)$ over $X\times Y$. On the set $X\vee Y=X\times v\cup v\times Y$, this section $s'$ could disagree with $s_Y\cup s_X$. We will correct it as follows.
By our assumption, $$\dim(X\vee Y)<d\cat(X\times Y)\le n.$$
We note that the fiber $F$ of the fibration $p$ is  homotopy equivalent to  the loop space $\Omega(X\vee Y)=\Omega(B\Gamma)$ where $\Gamma=\pi_1(X)\ast\pi_1(Y)$. Thus, $F$ is homotopy equivalent to the discrete space $\Gamma$. By Proposition~\ref{discrete},
the fiber $\mathcal B_{n+1}(F)$ of  $\mathcal B_{n+1}(p)$  is $(n-1)$-connected. Therefore, since $\dim(X\vee Y)<n$, there is a fiberwise deformation of $s'_{X\times Y}$ to a section $s_{X\times Y}$ that agrees with $s_Y\cup s_X$ (see Lemma 12 in~\cite{DS}).
Similarly, we construct a section $s_{Y\times X}$ that agrees with $s_Y\cup s_X$. Then $$s=s_Y\cup s_X\cup s_{X\times Y}\cup s_{Y\times X}$$ is a continuous section of $\mathcal B_{n+1}(p)$. By Proposition~\ref{chardTC}, $$\dTC(X\vee Y)\le n.$$
\end{proof}
\begin{rem} The equality in Theorem~\ref{wedge} is analogous to one from  Theorem 6 in~\cite{DS} with the  restrictions that $X$ and $Y$ are aspherical. We believe that this restriction can be dropped. On the other hand, this theorem is proven with the aim of applications to discrete groups whose classifying spaces are aspherical. 
\end{rem}

\subsection{Invariants for discrete groups}
Since $d\cat$ and $\dTC$ are homotopy invariant, they produce invariants of discrete groups. For a group $\Gamma$, define $d\cat(\Gamma)=d\cat(B\Gamma)$ and $\dTC(\Gamma)=\dTC(B\Gamma)$
where $B\Gamma=K(\Gamma,1)$ is a classifying space of the universal cover of spaces with the fundamental group $\Gamma$. By the Eilenberg-Ganea theorem~\cite{EG},
$\cat(\Gamma)=\cd(\Gamma)$, where $\cd$ is the cohomological dimension~\cite{Br}. In view of the equality $d\cat(\R P^\infty)=1$, this equality does not hold for $d\cat$, at least for groups with torsions. Still, for torsion-free groups, there is hope for the Eilenberg-Ganea type equality for $d\cat$\footnote{Recently Ben Knudson and Shmuel Weinberger
posted their preprint~\cite{KW} where they introduced a similar invariant $\acat$ for which they proved the equality $\acat(\Gamma)=\cd(\Gamma)$ for torsion free groups.
Their proof works for $d\cat$ as well.}.

However, like in the case of $\TC(\Gamma)$, the meaning of the value $\dTC(\Gamma)$ is unclear even for torsion-free groups. The following Proposition shows that like in the case of $\TC(\Gamma)$, the value $\dTC(\Gamma)$ can be anything between $\cd(\Gamma)$ and $\cd(\Gamma\times\Gamma)$.
\begin{prop}
$\dTC(\Z^m\ast\Z^n)=m+n$.
\end{prop}
\begin{proof}
Note that the natural classifying spaces for the groups $\mathbb{Z}^{m}$ and $\mathbb{Z}^{n}$ are the tori $T^m$ and $T^n$, respectively.
We assume that $m\ge n$. For $n=0$, since $T^m$ is a topological group, by Theorem~\ref{topgroup}, we have $\dTC(T^m)=d\cat(T^m)=m$, where the last equality follows from Proposition~\ref{spheres}. For $n>0$, we obtain $d\cat(T^m\times T^n)=m+n$ and the condition of Theorem~\ref{wedge} is satisfied. Then
$$
\dTC(T^m\vee T^n)=\max\{m, n, m+n\}=m+n.
$$
The observation that $T^m\vee T^n$ is a classifying space for the free product $\Z^m\ast\Z^n$ completes the proof.
\end{proof}

\section{Epilogue}\label{Epilogue}
As it was mentioned in section 3, there is a natural map $$\Phi:\mathcal B_n(P(X))\to P(\mathcal B_n(X))$$ defined as $\Phi(\sum\lambda_\phi\phi)(t)=\sum\lambda_\phi\phi(t)$.
The map $\Phi$ 
takes $\mathcal B_n(P(x,y))$ to $P(\delta_x,\delta_y)$.
We call a path $f:I\to \mathcal B_n(X)$ {\em resolvable} if  $f=\Phi(\mu)$ for some $\mu$. 
We define a navigation algorithm on $X$ by means of resolvable paths as a continuous map 
$$
m:X\times X\to P(\mathcal B_{n+1}(X))
$$
such that $m(x,y)$ is a resolvable path from $\delta_x$ to $\delta_y$ for each $(x,y)\in X\times X$.
It leads to the numerical invariant $\idTC(X)$ defined as the minimal $n$ such that $X$ admits the above navigation algorithm.
Clearly, $\idTC(X)\le\dTC(X)$. 
The difference between $\idTC$ and $\dTC$ is that in the latter a path from $x$ to $y$ can be uniquely defined as the image of $\le n+1$ distinct (unordered) strings with fixed distribution,
 whereas in the former case such presentation does exist but not unique, since the strings can intertwine
while keeping the distribution unchanged. By this reason
we call such numerical invariant the {\em intertwining distributional topological complexity} (idTC) .

Suppose that our robot (system) is so advanced that it not only can break into pieces before moving from position $x$ to position $y$ but its pieces can randomly rejoin and split during the motion. To cover this case we introduce the notion of locally resolvable path in $\mathcal B_n(X)$.
A path $f:I\to \mathcal B_n(X)$ is {\em resolvable at point} $t\in I$ if there exists $\epsilon>0$ such that the restrictions $f_{t+}=f|_{[t,t+\epsilon]}$ and $f_{t-}=f|_{[t-\epsilon,t]}$ treated as the elements of $P(\mathcal B_n(X))$ are resolvable. We note that this does not imply that the path $f_{[t-\epsilon,t+\epsilon]}$ is resolvable. This allows the distribution of mass vary
on the way from position $x$ to position $y$.
A path $f:I\to \mathcal B_n(X)$ is called {\em locally resolvable} if it is resolvable at each point $t\in I$.

In that case, a continuous navigation algorithm will be a continuous map
$$
v:X\times X\to P(\mathcal B_{n+1}(X))
$$
such that for all $x$,$y\in X$ the path $v(x,y)\in P(\delta_x,\delta_y)$  is locally resolvable.

The minimal number $n$ such that there exists such a continuous navigation algorithm $v$ is called  the \emph{varying distribution complexity}, denoted $\vdTC(X)$, of the configuration space $X$. Similarly, one can define the \emph{intertwining distributional LS-category}, $id\cat(X)$, and {\em varying distribution LS-category}, $vd\cat(X)$. 
We have   the following  inequalities:
$$\vdTC(X)\le\idTC(X)\le \dTC(X)$$ and $$vd\cat(X)\le id\cat(X)\le d\cat(X).$$

\end{document}